\documentclass[12pt,a4paper]{amsart}

\usepackage[margin=1in]{geometry}
\usepackage{amsmath,amsthm,amssymb}
\usepackage{mathtools}
\usepackage{wasysym}
\usepackage{listings}
\usepackage{todonotes}
\usepackage{hyperref}

\newtheorem{definition}{Definition}
\newtheorem{theorem}{Theorem}
\newtheorem*{theorem*}{Theorem}
\newtheorem{conjecture}{Conjecture}
\newtheorem{remark}{Remark}
\newtheorem{observation}{Observation}
\newtheorem{proposition}{Proposition}
\newtheorem*{proposition*}{Proposition}
\newtheorem{lemma}{Lemma}
\newtheorem*{lemma*}{Lemma}
\newtheorem*{claim*}{Claim}
\newtheorem{corollary}{Corollary}
\DeclareMathOperator{\const}{const}

\allowdisplaybreaks

\lstdefinestyle{mystyle}{
    basicstyle=\ttfamily\footnotesize
}
\begin{document}
\title{An asymptotic lower bound on the number of polyominoes}
\author{Vuong Bui}
\address{Vuong Bui, Institut f\"ur Informatik, Freie Universit\"{a}t
Berlin, Takustra{\ss}e~9, 14195 Berlin, Germany,
LIRMM, Universit\'e de Montpellier, 161 Rue Ada, 34095 Montpellier, France, and UET, Vietnam National University, Hanoi, 144 Xuan Thuy Street, Hanoi 100000, Vietnam}
\email{bui.vuong@yandex.ru}

\begin{abstract}
    Let $P(n)$ be the number of polyominoes of $n$ cells and $\lambda$ be Klarner's constant, that is, $\lambda=\lim_{n\to\infty} \sqrt[n]{P(n)}$. We show that there exist some positive numbers $A,T$, so that for every $n$
    \[
        P(n) \ge An^{-T\log n} \lambda^n.
    \]
    This is somewhat a step toward the well known conjecture that there exist positive $C,\theta$ so that $P(n)\sim Cn^{-\theta}\lambda^n$ for every $n$. In fact, if we assume another popular conjecture that $P(n)/P(n-1)$ is increasing, we can get rid of $\log n$ to have 
    \[
        P(n)\ge An^{-T}\lambda^n.
    \]
    
    Beside the above theoretical result, we also conjecture that the ratio of the number of some class of polyominoes, namely inconstructible polyominoes, over $P(n)$ is decreasing, by observing this behavior for the available values. The conjecture opens a nice approach to bounding $\lambda$ from above, since if it is the case, we can conclude that
    \[
        \lambda < 4.1141,
    \]
    which is quite close to the current best lower bound $\lambda > 4.0025$ and greatly improves the current best upper bound $\lambda < 4.5252$.
    
    The approach is merely analytically manipulating the known or likely properties of the function $P(n)$, instead of giving new insights of the structure of polyominoes. The techniques can be applied to other lattice animals and self-avoiding polygons of a given area with almost no change.
\end{abstract}
\maketitle

\section{Introduction}
Polyomino is a popular type of geometric figures and it can be recognized by the example in Fig. \ref{fig:polyomino-lattice}. Rigorously speaking, a polyomino is a finite subset of the points in the square lattice whose induced graph is connected. In fact, when the lattice is not the square lattice, the corresponding set has another specific name (e.g. polyiamond for the honeycomb lattice). In general, all of them are called lattice animals.
In this article, we study only polyominoes. It means we work on the square lattice, but we will depict polyominoes in the equivalent form with cells instead of points, which is perhaps more convenient to the readers.

\begin{figure}[ht]
    \includegraphics[width=0.4\textwidth]{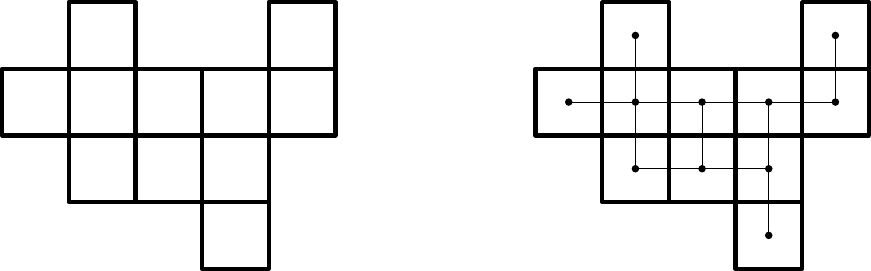}
    \caption{A polyomino with its underlying presentation in the square lattice}
    \label{fig:polyomino-lattice}
\end{figure}

There are various articles, books on the topic of polyominoes. For an up-to-date reference on the subject with research problems, we recommend the readers to check the chapter ``Polyominoes'' by Barequet, Golomb, and Klarner in the book \cite{toth2017handbook}. This article studies the asymptotic behaviors of the number of polyominoes.

We denote by $P(n)$ the number of polyominoes of $n$ cells. Note that we count equivalent polyominoes as one polyomino only. Here, polyominoes are said to be equivalent if they are translates of each other. In literature, other operations are also considered, e.g. congruent operations including translation, rotation, and reflection. However, the equivalence by translation is probably the most popular.

We will briefly introduce the origin of the growth constant of $P(n)$ and a well known conjecture on the estimation of $P(n)$. After that, we will give partial results toward the conjecture, and propose a conditional upper bound on the growth constant as well.

\subsection*{Growth constant and a conjecture on the order}
Let us define a lexicographical order on the cells: A cell $c_1$ is said to be smaller than another cell $c_2$ if either $c_1$ is on a column to the left of $c_2$, or the two cells are on the same column and $c_1$ is under $c_2$.
 
Given two polyominoes of $\ell,m$ cells, we can concatenate the two polyominoes to obtain a unique polyomino of $\ell+m$ cells by placing the largest cell of the first polyomino right under the smallest cell of the second polyomino. An example is given in Fig. \ref{fig:polyomino-concatenation}, where we concatenate the polyomino in Fig. \ref{fig:polyomino-lattice} with a translate of itself (the order of the cells in the first polyomino is also depicted). This type of concatenations gives the following result.

\begin{figure}[ht]
    \includegraphics[width=0.3\textwidth]{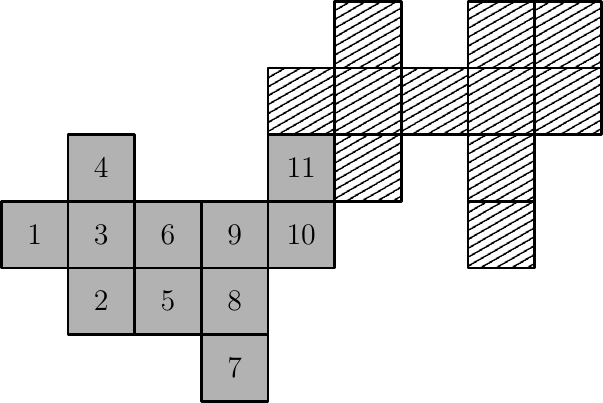}
    \caption{A polyomino concatenated with a translate of itself}
    \label{fig:polyomino-concatenation}
\end{figure}

\begin{lemma}[Klarner 1967 \cite{klarner1967cell}] \label{lem:supermulti}
    The sequence $P(n)$ is supermultiplicative, that is for any $\ell,m$, we have
    \[
        P(\ell+m)\ge P(\ell) P(m).
    \]
\end{lemma}
The lemma was first used by Klarner to prove that the limit of $\sqrt[n]{P(n)}$ exists, by using Fekete's lemma \cite{fekete1923verteilung} for supermultiplicative sequences. It is Klarner after whom the following constant is named.
\begin{definition}[Klarner's constant]
    The growth constant of the polyominoes is the limit
    \[
        \lambda\coloneqq\lim_{n\to\infty} \sqrt[n]{P(n)}.
    \]
\end{definition}

Beside the existence of the limit, Fekete's lemma also states that $\lambda=\sup_n \sqrt[n]{P(n)}$, which means
\[
    P(n)\le \lambda^n.
\]
By the available values of $P(n)$, it seems that $P(n)$ is quite lower than $\lambda^n$, probably by a polynomial number of times. In fact, the following conjecture on the estimation of $P(n)$ is widely believed in literature, and has stood open since as far as 1976 \cite{gaunt1976percolation,sykes1976percolation}.
\begin{conjecture}%[Whittington and Soteros \cite{whittington1990lattice}] 
\label{conj:order}
    There exist constants $C,\theta$ so that
    \[
        P(n) \sim C n^{-\theta} \lambda^n.
    \]
\end{conjecture}
In various sources, e.g. \cite{jensen2000statistics}, it is even believed that $\theta=1$. However, it seems that no attempt has been made to settle down either conjecture, other than the evidence by the available values of $P(n)$.

\subsection*{A theoretical lower bound on $P(n)$}
We will give a step toward Conjecture \ref{conj:order} in Theorem \ref{thm:lower-bound}, by using only two functional properties of $P(n)$ instead of the actual values of $P(n)$ or any other insights on the structure of polyominoes. One property is actually the supermultiplicativity in Lemma \ref{lem:supermulti}. The other is in Lemma \ref{lem:supporting-upper-bound} below. At first, we give the origin of the lemma by the following notion of composition, whose instances include the concatenation for Lemma \ref{lem:supermulti}.

\begin{definition}
    A composition of two polyominoes is the union of some translates of the polyominoes so that the translates are disjoint (except at the edges of the cells) and the union of the translates is connected. 
\end{definition}
Fig. \ref{fig:polyomino-composition} illustrates an example of the compositions of two polyominoes.
\begin{figure}[ht]
    \includegraphics[width=0.5\textwidth]{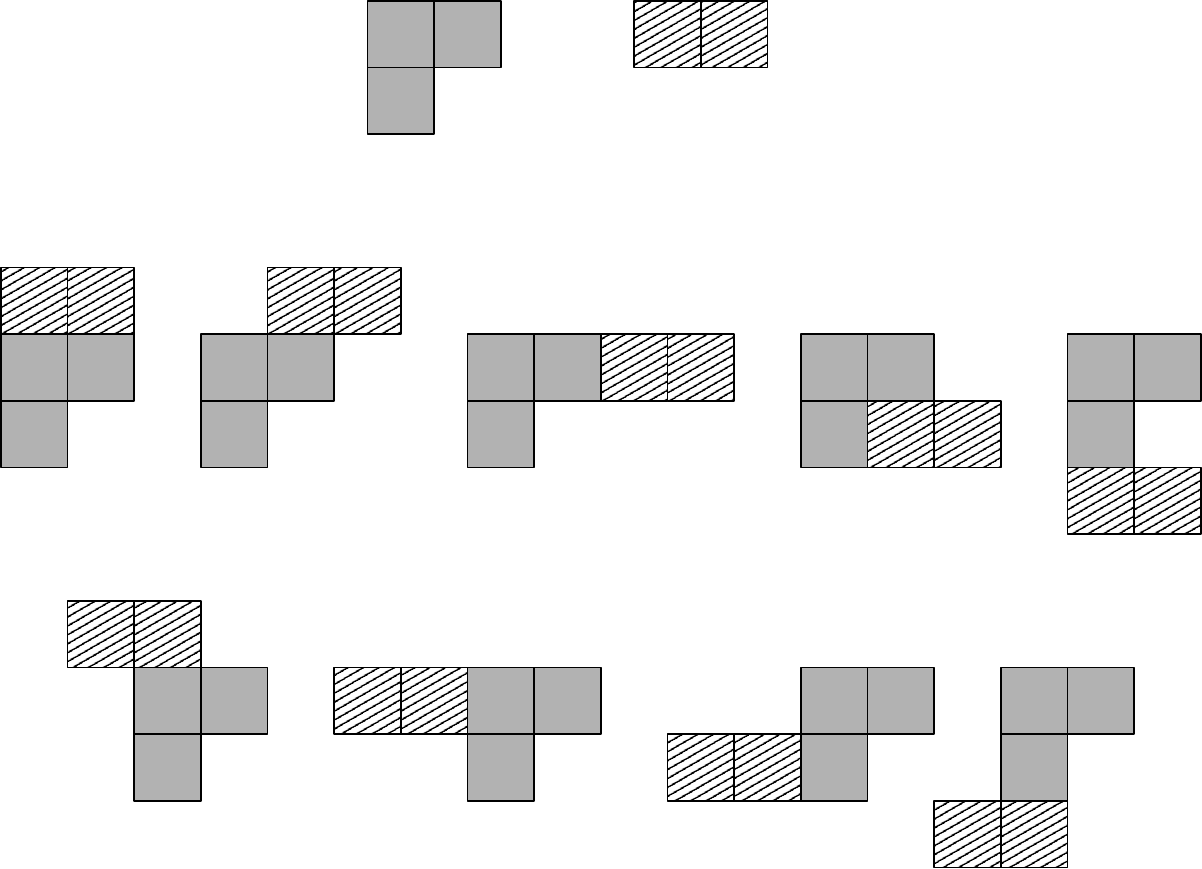}
    \caption{Compositions of two polyominoes of $3$ cells and $2$ cells}
    \label{fig:polyomino-composition}
\end{figure}

A composition of two polyominoes $X,Y$ of $\ell,m$ cells is a polyomino of $\ell+m$ cells. To let two translates be connected, at least one cell $x$ of $X$ touches at least one cell $y$ of $Y$. There are at most $4$ ways to do that: either $x$ to the left of $y$, or to the right of $y$, or above $y$, or below $y$. Therefore, we have the following observation.
\begin{observation}
    There are at most $4\ell m$ compositions of any two given polyominoes of $\ell,m$ cells.
\end{observation}

\begin{remark}
   Two compositions can be actually the same polyomino (or more precisely translates of each other). Therefore, the number of compositions can be in principle quite lower than $4\ell m$, e.g. the number of compositions is of only linear order for two vertical or horizontal bars. The order can be as low as $\sqrt{n}$ when $\ell=m=n$, which is attained by two square blocks. In \cite{asinowski2021number}, it is shown that the order cannot be lower than $\sqrt{n}$, and an interesting example shows that the order can be as high as $\frac{n^2}{2^{8(\sqrt{\log_2 n})}}$, which is very close to the bound of order $n^2$ in the observation. It is natural to ask if the order $n^2$ is attainable, for which Rote has suggested that it is not possible (private communication). It seems to be a good research direction, that is why we give more background than necessary.
\end{remark}

The following proposition is a useful result about compositions.
\begin{proposition}[A slight modification\footnote{The original condition is $\lceil (n-1)/4\rceil \le \ell \le \lfloor (3n+1)/4\rfloor$, but we find working directly with the fractions more convenient later.} of Barequet and Barequet \cite{barequet2015improved}]
\label{prop:summand-size}
    Every polyomino of $n\ge 2$ cells is a composition of two polyominoes of $\ell,n-\ell$ cells for some $\ell$ satisfying $(n-1)/4 \le \ell \le (3n+1)/4$.
\end{proposition}
\begin{proof}
    We treat the polyomino as a connected graph whose vertices are the cells and the edges are the adjacencies between the cells. Each vertex has the degree at most $4$. Consider a rooted spanning tree. If any subtree has the number of vertices in the range $[(n-1)/4, (3n+1)/4]$, then we are done since one of the two polyominoes is the polyomino induced by the subtree, and the other polyomino induced by the remaining vertices is also connected.
    
    In the remaining case when there is no such subtree, we consider a maximal subtree that has more than $(3n+1)/4$ vertices (in the sense that none of its proper subtrees satisfies the properties). If the root of the subtree has $4$ subtrees, then the subtree is actually the whole tree of $n$ leaves itself. We have a contradiction, since the total number of vertices in the subtrees of the root is less than $4(n-1)/4 = n-1$. If the root of the subtree has at most $3$ subtrees, we have again a contradiction, since the total number of vertices in these subtrees is less than $3(n-1)/4 = (3n+1)/4 - 1$.
\end{proof}
\begin{remark}\label{rem:variant-in-size}
        Although we would keep the proposition as the original version in \cite{barequet2015improved}, we can actually narrow down the interval of $\ell$ to $[(n-1)/4,n/2]$ by assuming that $\ell\le n-\ell$ (since the roles of $\ell,n-\ell$ are exchangeable). Lemma \ref{lem:supporting-upper-bound} employs this trick.
\end{remark}

The following lemma, which is the other key lemma beside Lemma \ref{lem:supermulti}, is a corollary of Proposition \ref{prop:summand-size}.
\begin{lemma} \label{lem:supporting-upper-bound}
    For every $n$, there exists some $\ell$ so that $(n-1)/4 \le \ell \le n/2$ and
    \[
        P(n)\le n^3 P(\ell) P(n-\ell).
    \]
\end{lemma}
\begin{proof}
    There are at most $4xy$ compositions of two polyominoes of $x,y$ cells. Therefore, it follows from the variant of Proposition \ref{prop:summand-size} in Remark \ref{rem:variant-in-size} that
    \begin{align*}
        P(n)&\le \sum_{(n-1)/4 \le \ell' \le n/2} 4\ell'(n-\ell') P(\ell')P(n-\ell') \\
        &\le \sum_{(n-1)/4 \le \ell' \le n/2} (\ell'+n-\ell')^2 P(\ell')P(n-\ell') \\
        &\le n^3 P(\ell)P(n-\ell)
    \end{align*}
    for some $\ell$ that maximizes $P(\ell)P(n-\ell)$.
\end{proof}
\begin{remark}
    In fact, we can have a better leading constant than $1$ in the upper bound, since the number of integers $\ell'$ in range $[(n-1)/4, n/2]$ is about $n/4$ instead of $n$. One may apply further optimizations here. However, we keep the leading constant simple, since the higher priority is to optimize the degree of the leading polynomial $n^3$ instead, as discussed in the remark below. 
\end{remark}
\begin{remark}
    If Conjecture \ref{conj:order} holds, then it follows from Lemma \ref{lem:supporting-upper-bound} that the value of $\theta$ in the conjecture is at most $3$. However, it is a rather poor bound (note that it is believed to be $1$ by empirical evidence). The trivial lower bound is $0$, due to Lemma \ref{lem:supermulti}. A more advanced lower bound on $\theta$ is $1/2$ due to Madras \cite{madras1995rigorous}:
    \[
        P(n)\le \frac{1}{\sqrt{8}} n^{-1/2}\lambda^n.
    \]
    An open problem is to decrease the degree $3$. Note that the constant $3$ appears many times in the article. (Another reason is due to the denominator $3$ in Definition \ref{def:weak-submulti} in Section \ref{sec:lower-bound}.)
\end{remark}
Although $P(n)$ has the descriptive meaning of counting polyominoes, the functional properties in Lemma \ref{lem:supermulti} and Lemma \ref{lem:supporting-upper-bound} are sufficient for proving the following asymptotic lower bound of $P(n)$. Since not many functional properties of $P(n)$ are known, the theorem below can be seen as a summary of the current knowledge and we kind of have an idea of how far we are from the conjecture.

\begin{theorem} \label{thm:lower-bound}
    There exist positive numbers $A,T$ so that for every $n$,
    \[
        P(n)\ge An^{-T\log n} \lambda^n.
    \]
    In particular, for every $n$,
    \[
        P(n)\ge 3^{-18\alpha\log 3} n^{-12\alpha\log 3} n^{ -3\alpha\log n}\lambda^n
    \]
    where $\alpha=\frac{3}{\log\frac{16}{5}}$.
\end{theorem}
Note that we denote by $\log x$ the binary logarithm $\log_2 x$ for the whole article.

Rounding the constants, we have $\alpha\approx 1.7878$ and
\[
    P(n)\ge 3^{-51.0038} n^{-34.0026} n^{5.3633\log n} \lambda^n.
\]

The belief in Conjecture \ref{conj:order} is a bit improved since
\[
    An^{-T\log n} \lambda^n \le An^{-T} \lambda^n \le \lambda^n.
\]

In \cite{van1992number}, a similar result\footnote{In \cite{van1992number}, high dimensional lattices are also treated, but we only bring the results for $d=2$ to compare.}  is known for the number of trees $t_n$ weakly embedded in the hypercubic lattice: There exist positive $A^*,T^*$ so that for every $n$, we have $t_n\ge A^*n^{-T^*\log n} (\lambda^*)^n$ where $\lambda^*$ is the growth rate of the number of trees. However, the estimation there (for $d=2)$ has weaker constants:
\begin{equation}\label{eq:tree-lower-bound}
    t_n\ge e^{-24\alpha^*} n^{-24\alpha^*} e^{-9\alpha^*(\log n)^2} (\lambda^*)^n,
\end{equation}
where $\alpha^*=\frac{1}{\log\frac{4}{3}}$. (Note that $\alpha<\alpha^*$.)

If we wish for the strong form $t_n\ge A^*n^{-T^*}(\lambda^*)^n$ with probably some negotiation in the values of $n$, \cite[Equation ($24$)]{madras2017location} states that
\[
    t_n\ge n^{-4}(\lambda^*)^n
\]
for \emph{infinitely many} values of $n$.
On the other hand, we show in Theorem \ref{thm:conditional-lower-bound} that $P(n)\ge 3^{-18}n^{-9}\lambda^n$ for all $n$ \emph{provided that} Conjecture \ref{conj:increasing} on the increasing monotonicity of $P(n)/P(n-1)$ holds.

The merit\footnote{In fact, the author was at first not aware of \cite{van1992number}, which is fortunately suggested by an anonymous reviewer.} of the article lies in analytically manipulating the properties of the function $P(n)$ while even not giving any significantly new properties of $P(n)$.
More comparisons to the work \cite{van1992number} in the approaches are given before the proof of Theorem \ref{thm:lower-bound}.

\begin{remark}
   In fact, the natures of trees and polyominoes are quite similar. While decomposing the trees is quite straightforward, we can apply the same thing with a spanning tree of the underlying graph of the polyominoes, as in the proof of Proposition \ref{prop:summand-size}. 
On the other hand, other kinds of graphs may be hard to deal with. For example, for the number of self-avoiding walks $c_n$ with the length $n$, we only know by the result of Hammersley and Welsh \cite{hammersley1962further} that
\[
    \mu^n\le c_n\le \mu^n e^{\kappa\sqrt{n}},
\]
where $\mu$ is the growth rate and $\kappa$ is some constant. This result has stood unimproved since $1962$. The similar situation applies to the number of self-avoiding polygons $q_{2n}$ with the perimeter $2n$. It is computed by $q_{2n}=\frac{c_{2n-1}(e_1)}{n}$, where $c_{2n-1(e_1)}$ is the number of self-avoiding walks of length $2n$ returning to the origin with an edge of direction $e_1$, and
\[
    \mu^{2n} e^{-C \sqrt{n}} \le c_{2n+1}(e_1) \le (n+1) \mu^{2n+2},
\]
where $C$ is a constant and the growth rate here $\mu$ is the same as the growth rate of the number of self-avoiding walks. Note that the same approach in this work can apply (with almost no change) to the number of self-avoiding polygons of a given area, since they are equivalent to polyominoes that do not have holes. (The approach in this work is not affected by holes.)
\end{remark}

\subsection*{A conditional lower bound on $P(n)$}
It seems that getting rid of the factor $\log n$ in $P(n)\ge An^{-T\log n}\lambda^n$ is hard. However, we can achieve it provided that the following conjecture holds. It is quite popular and suggested by the available values of $P(n)$.
\begin{conjecture} \label{conj:increasing}
    The sequence $P(n)/P(n-1)$ is monotonically increasing.
\end{conjecture}
The origin of Conjecture \ref{conj:increasing} is hard to trace,\footnote{An anonymous reviewer of this paper suggests that Conjecture \ref{conj:increasing} is a consequence of the ratio method of series analysis, due to Sykes and Domb, developed in the 1950s and first given in the PhD thesis of Sykes. 
However, we have unfortunately no access to the thesis.} but a reference is Chapter $14$ by Barequet, Golomb, and Klarner about polyominoes in the book \cite{toth2017handbook} (Problem $14.3.7$). 

\begin{theorem} \label{thm:conditional-lower-bound}
    If Conjecture \ref{conj:increasing} holds, then there exist positive numbers $A,T$ so that for every $n$,
    \[
        P(n)\ge An^{-T}\lambda^n.
    \]
    In particular, for every $n$,
    \[
        P(n)\ge 3^{-18} n^{-9}\lambda^n.
    \]
\end{theorem}

We should note that proving Conjecture \ref{conj:increasing} is probably hard.
While the proof that the limit of $\sqrt[n]{P(n)}$ exists is fairly straightforward by the supermultiplicativity of $P(n)$ in Lemma \ref{lem:supermulti}, another interesting asymptotic result, which is even weaker than Conjecture \ref{conj:increasing}, takes a quite large effort.
\begin{theorem*}[Madras 1999 \cite{madras1999pattern}]
    The limit of $P(n)/P(n-1)$ as $n\to\infty$ exists. 
\end{theorem*}
The limit is obviously identical to Klarner's constant.

Suppose Conjecture \ref{conj:increasing} is true, we immediately have $P(n)/P(n-1)$ for any $n$ is a lower bound of $\lambda$. With the available values of $P(n)$ in \cite{jensen2003counting}, for $n=56$, we have
\[
    \lambda\ge\frac{P(56)}{P(55)} = \frac{69150714562532896936574425480218}{17326987021737904384935434351490} > 3.9909,
\]
which is a fairly good result.\footnote{In a private correspondence with Gill Barequet, the author was informed of the recent advance by Barequet and Ben-Shachar in calculating values of $P(n)$ till $n=70$, for which we have
\[
    \lambda\ge\frac{P(70)}{P(69)} = \frac{18500792645885711270652890811942343400814}{4619282047583828929546825973053580643926} > 4.0051,
\]
which is even better than the current best lower bound $4.0025$.} The best known lower bound so far is $\lambda> 4.0025$ (see \cite{barequet2016lambda}), but it uses somewhat more computationally intensive techniques. (Fairly speaking, computing $P(56)$ is also computationally intensive, but we are provided this resource in the first place though. Most probably Conjecture \ref{conj:increasing} is initially suggested by the available values of $P(n)$.)

An interesting point is that Conjecture \ref{conj:increasing} also supports establishing a conditional upper bound on $\lambda$, by Theorem \ref{thm:conditional-lower-bound}. However, the constants in $3^{-18} n^{-9}$ makes the estimates very poor. That is the reason why we will devote Section \ref{sec:generating-function} for another approach to an effective conditional upper bound on $\lambda$.

\section{An asymptotic lower bound on $P(n)$}
\label{sec:lower-bound}
This section proves Theorem \ref{thm:lower-bound}. Before doing that, we present the following mediate step.

\begin{proposition} \label{prop:supporting-upper-bound}
    For any $n\ge 1$ and any $\ell$ with $0\le \ell\le n$, we have
    \[
        P(n)\le n^{3\alpha\log n} P(\ell) P(n-\ell)
    \]
    where $\alpha=\frac{3}{\log\frac{16}{5}}$.
\end{proposition}
Note that we assume $P(0)=1$ in case $\ell=0$ or $\ell=n$.

The reason we do not write $n^{\alpha\log n}$ and $\alpha=\frac{9}{\log \frac{16}{5}}$ is that we will compare this $\alpha=\frac{3}{\log \frac{16}{5}} \approx 1.7877$ with the constant $\alpha^*=\frac{1}{\log \frac{4}{3}} \approx 2.4094$ in the work \cite{van1992number} about the number of trees weakly embedded in the lattice $t_n$: 
\begin{equation}\label{eq:tree-upper-bound}
    t_n\le n^{3\alpha^*\log n} t_{\ell} t_{n-\ell}.
\end{equation}
Note that the constant $3$ in $3\alpha$ is from the degree $3$ in Lemma \ref{lem:supporting-upper-bound}.

One would say that the constant $\alpha$ is better than the constant $\alpha^*$. To prove Proposition \ref{prop:supporting-upper-bound} (or the inequality \eqref{eq:tree-upper-bound}), one needs to partition the number $n$ into smaller parts, and recursively partition them further, and finally combine the parts into two collections of size $n$ and $n-\ell$, with the support of Lemma \ref{lem:supermulti}. The fraction $\frac{4}{3}$ in $\alpha^*$ suggests that in each iteration, the proportion of the part be bounded by $\frac{3}{4}$, which is also guaranteed for polyominoes by Proposition \ref{prop:summand-size}. However, the tricky part is that sometimes, the fraction may be smaller than the bound $\frac{3}{4}$, and if we look further into the next iterations, we can obtain better bounds on the fractions. That is the merit of the approach in the proof below, which is actually quite analytically complicated for that reason. (The readers can check for themselves that the proof becomes very simple when we just need $\alpha=\alpha^*=\frac{1}{\log \frac{4}{3}}$.)

The proof uses induction and the base case is $n<8$. One may simply use a computer to confirm it with all choices of $\ell$ for each $n$, but there is a trick to reduce the workload. Since Conjecture \ref{conj:increasing} holds for the available values of $P(n)$, we can apply Observation \ref{obs:same-sum} (in Section \ref{sec:conditional-lower-bound}) for $n<8$ and obtain that the value of $P(\ell) P(n-\ell)$ gets minimal when $\ell=\lfloor n/2\rfloor$. Therefore, we only need to compare $P(n)$ with $\hat P(n)=n^{3\alpha\log n} P(\lfloor n/2\rfloor) P(\lceil n/2\rceil)$ as in the following table. (Note that we only approximate from below the value of $\hat P(n)$ by an integer in the table.) 
\begin{center}
    \begin{tabular}{|c|c|c|c|c|c|c|c|c|}
        \hline
         n&1&2&3&4&5&6&7 \\
         \hline
         P(n) & 1& 2& 6& 19& 63& 216& 760 \\
         \hline
         $\lfloor \hat P(n)\rfloor$& 1 & 41 & 22743 & 11484563 & 6075603512 & 2210715359865 & 604333850455172 \\
         \hline
    \end{tabular}
\end{center}
\begin{proof}
    The conclusion holds for $n<8$ as discussed above. Given some $n\ge 8$, suppose the conclusion holds for any number smaller than $n$, we show that it also holds for $n$.
    
    Since $\ell$ and $n-\ell$ are interchangeable, we assume that $\ell\le n-\ell$, that is $\ell\le n/2$.
    
    By Lemma \ref{lem:supporting-upper-bound}, there exists $k$ so that $(n-1)/4 \le k\le n/2$ and
    \begin{equation}\label{eq:first-eq}
        P(n)\le n^3 P(k) P(n-k).
    \end{equation}

    If $k\ge \ell$, then applying the induction hypothesis to $P(k)$ in \eqref{eq:first-eq}, we have
    \begin{align*}
        P(n)&\le n^3 k^{3\alpha\log k} P(\ell) P(k-\ell) P(n-k) \\
        &\le n^3 n^{3\alpha\log k} P(\ell)P(n-\ell) \\
        &\le n^{3\alpha\log n} P(\ell)P(n-\ell),
    \end{align*}
    where $P(k-\ell)P(n-k)\le P(n-\ell)$ is due to the supermultiplicativity of $P(n)$ in Lemma \ref{lem:supermulti}.
    Note that 
    \[
        3+3\alpha\log k\le 3\alpha\left(\frac{1}{\alpha}+\log\frac{n}{2}\right)=3\alpha\left(\frac{\log\frac{16}{5}}{3}+(\log n) - 1\right)\le 3\alpha\log n.
    \]

    Consider the case $k<\ell$.
    
    Applying Lemma \ref{lem:supporting-upper-bound} to $P(n-k)$ in \eqref{eq:first-eq}, there exists some $a$ so that $(n-k-1)/4\le a\le (n-k)/2$ and
    \begin{equation}\label{eq:second-eq}
        P(n)\le n^3P(k)(n-k)^3 P(a)P(n-k-a)\le n^6 P(k) P(a) P(n-k-a).
    \end{equation}

    If $a\ge \ell-k$ then applying the induction hypothesis to $P(a)$ in \eqref{eq:second-eq}, we have
    \begin{align*}
        P(n)&\le n^6 P(k) a^{3\alpha\log a} P(\ell-k) P(a-(\ell-k)) P(n-k-a) \\
        &\le n^6 n^{3\alpha\log a} [P(k)P(\ell-k)][P(a-(\ell-k)) P(n-k-a)] \\
        &\le n^{3\alpha\log n} P(\ell)P(n-\ell).
    \end{align*}
    In the last inequality, we continue to apply the supermultiplicativity of $P(n)$ in Lemma $1$ to the pairs in square brackets. Also note that 
    \begin{gather*}
        6+3\alpha\log a \le 3\alpha\left(\frac{2}{\alpha}+\log \frac{n-k}{2}\right) \le 3\alpha\left(\frac{2}{\alpha}+\log \left(\frac{1}{2}\cdot \frac{3n+1}{4}\right)\right) \\
        =   3\alpha\left(\frac{2}{\alpha} - 3 + (\log n) + \log \left(3+\frac{1}{n}\right)\right) \le 3\alpha\left(\log n + \frac{2\log\frac{16}{5}}{3} - 3 + \log \left(3+\frac{1}{2}\right)\right) \\
        \le 3\alpha\log n
    \end{gather*}
    since $n-k\le \frac{3n+1}{4}$ and $n\ge 2$.

    Consider the case $a<\ell-k$.
    Applying Lemma \ref{lem:supporting-upper-bound} to $P(n-k-a)$ in \eqref{eq:second-eq}, there exists some $b$ so that $(n-k-a-1)/4\le b\le (n-k-a)/2$ and
    \begin{equation}\label{eq:third-eq}
        P(n)\le n^6 P(k) P(a) (n-k-a)^3 P(b) P(n-k-a-b) \le n^9 P(k+a) P(b) P(n-k-a-b).
    \end{equation}
    
    Since $b\ge (n-k-a-1)/4$ and $n\ge 8$, it follows that $b\ge \ell-k-a$. Indeed, 
    \begin{gather*}
        (n-k-a-1)-4(\ell-k-a)=(n-4\ell)+3k+3a-1\ge (n-4\frac{n}{2}) + 3k+3\frac{(n-k-1)}{4}-1\\
        =-\frac{n}{4} + \frac{9}{4}k - \frac{7}{4} \ge -\frac{n}{4} + \frac{9}{4}\cdot \frac{n-1}{4} - \frac{7}{4} = \frac{1}{16}(5n-37)\ge 0.
    \end{gather*}

    We also observe that
    \begin{gather*}
        b\le\frac{n-k-a}{2}\le\frac{n-k}{2} - \frac{1}{2}\cdot \frac{n-k-1}{4} = \frac{3(n-k)}{8} + \frac{1}{8} \\
        \le \frac{3(n-\frac{n-1}{4})}{8} + \frac{1}{8} = \frac{9n+7}{32} \le \frac{10n}{32} = \frac{5n}{16},
    \end{gather*}
since $n\ge 7$.
    
    Applying the induction hypothesis to $P(b)$ in \eqref{eq:third-eq}, we have
    \begin{align*}
        P(n)&\le n^9 P(k+a) b^{3\alpha\log b} P(\ell-k-a) P(b-(\ell-k-a)) P(n-k-a-b)\\
        &\le n^9 n^{3\alpha\log b} [P(k+a) P(\ell-k-a)][P(b-(\ell-k-a)) P(n-k-a-b)]\\
        &\le n^{3\alpha\log n} P(\ell)P(n-\ell),
    \end{align*}
    since 
    \[
        9+3\alpha\log b=3\alpha\left(\frac{3}{\alpha} + \log b\right)\le 3\alpha\left(\log \frac{16}{5} + \log \left(n\frac{5}{16}\right)\right) = 3\alpha\log n.
    \]
 
    The conclusion follows by induction.
\end{proof}

Although Fekete's lemma is originally stated for submultiplicative sequences, it can be applied to the following weak form of submultiplicativity.
\begin{definition}[Weak form of submultiplicativity] \label{def:weak-submulti}
    A positive sequence $s_n$ is said to be weakly submultiplicative if for every $\ell,m,n$ so that $\ell+m=n$ and $\ell,m\in \left[\frac{n}{3}, \frac{2n}{3}\right]$, we have $s_n\le s_\ell s_m$.
\end{definition}
At first, we note that submultiplicativity and supermultiplicativity can be seen as the dual forms of each other: A positive sequence $s_n$ is supermultiplicative if and only if the sequence $s'_n=1/s_n$ is submultiplicative. The limit value $\sup_n\sqrt[n]{s_n}$ in the former case can be transformed into $\inf_n\sqrt[n]{s'_n}$ for the latter case.
Secondly, the idea of conditioning the ratio of $\ell/m$ is due to the following well known extension of Fekete's lemma.
\begin{lemma*}[Extension of Fekete's lemma]
    If $s_n$ is weakly submultiplicative, then the following limit exists and can be written as:
    \[
        \lim_{n\to\infty} \sqrt[n]{s_n} = \inf_n \sqrt[n]{s_n}.
    \]
\end{lemma*}
The readers may try to adjust the original proof of Fekete's lemma for the extension, or check out \cite[Theorem $22$]{de1952some} for a proof.

Before proving the lower bound of $P(n)$ in Theorem \ref{thm:lower-bound}, we would remark that our approach here is different from the approach in \cite{van1992number} for proving \eqref{eq:tree-lower-bound}, although both use similar supporting results, which are Proposition \ref{prop:supporting-upper-bound} and \eqref{eq:tree-upper-bound}. The approach in \cite{van1992number} uses a result by Hammersley \cite{hammersley1962generalization} that: If $t_{n+m}\le h(n+m)t_nt_m$ then 
\[
    \frac{\log t_n}{n}\ge \log \lambda^* + \frac{\log h(n)}{n} - 4\sum_{m=2n}^\infty \frac{\log h(m)}{m(m+1)}.
\]
Our approach is more self-contained and appears to be more elementary: We transform $P(n)$ into a function $g(n)$ so that $g(n)$ is weakly submultiplicative and $g(n)$ has the same growth rate as $P(n)$. This approach, however, gives better constants in the final result than the ones in \eqref{eq:tree-lower-bound}.

\begin{proof}[Proof of Theorem \ref{thm:lower-bound}]
    It suffices to prove that for every $n$,
    \[
        P(n)\ge 3^{-18\alpha\log 3} n^{-12\alpha\log 3} n^{ -3\alpha\log n}\lambda^n
    \]
    where $\alpha=\frac{3}{\log\frac{16}{5}}$, since the existence of $A,T$ so that $P(n)\ge An^{-T\log n}$ follows.
    
    By Proposition \ref{prop:supporting-upper-bound}, for every $\ell,m,n$ so that $\ell+m=n$, we have
    \[
        P(n)\le n^{3\alpha \log n} P(\ell) P(m).
    \]
    
    For a pair $\ell,m$ so that both $\ell,m$ are in the range $\left[\frac{n}{3}, \frac{2n}{3}\right]$, we can write
    \begin{align*}
        n^{3\alpha\log n} P(n) & \le n^{3\alpha\log n} P(\ell) n^{3\alpha\log n} P(m) \\
        & \le \left(\frac{n}{\ell}\ell\right)^{3\alpha\log(\frac{n}{\ell}\ell)} P(\ell) \left(\frac{n}{m}m\right)^{3\alpha\log(\frac{n}{m}m)} P(m) \\
        & \le \left(\frac{n}{\ell}\right)^{3\alpha\log\frac{n}{\ell}} \left(\frac{n}{\ell}\right)^{3\alpha\log\ell} \ell^{3\alpha\log\frac{n}{\ell}} \ell^{3\alpha\log\ell}  P(\ell) \times \\
            &\qquad \left(\frac{n}{m}\right)^{3\alpha\log\frac{n}{m}} \left(\frac{n}{m}\right)^{3\alpha\log m} m^{3\alpha\log\frac{n}{m}} m^{3\alpha\log m}  P(m) \\
        & \le 3^{3\alpha\log 3} 3^{3\alpha\log\ell} \ell^{3\alpha\log 3} \left[\ell^{3\alpha\log\ell} P(\ell)\right] \times \\
            &\qquad 3^{3\alpha\log 3} 3^{3\alpha\log m} m^{3\alpha\log 3} \left[m^{3\alpha\log m} P(m)\right] \\
        & \le (3^{3\alpha\log 3})^2 (\ell^{3\alpha\log 3})^2 \left[\ell^{3\alpha\log\ell} P(\ell)\right] (m^{3\alpha\log 3})^2 \left[m^{3\alpha\log m} P(m)\right], %\\
    \end{align*}
    where the last inequality is due to $x^{\log y}=y^{\log x}$ for any $x,y$.
    
    Let us rewrite this by
    \[
        f(n)\le 3^\beta \ell^\beta f(\ell) m^\beta f(m),
    \]
    where $\beta=6\alpha\log 3$ and $f(n)=n^{3\alpha\log n} P(n)$.

    We further have
    \begin{align*}
        3^{3\beta} n^{2\beta} f(n) &\le 3^{2\beta} n^\beta \ell^\beta f(\ell) 3^{2\beta} n^\beta m^\beta f(m) \\
         &\le 3^{2\beta} (3\ell)^\beta \ell^\beta f(\ell) 3^{2\beta} (3m)^\beta m^\beta f(m) \\
        &= \left[3^{3\beta} \ell^{2\beta} f(\ell)\right] \left[3^{3\beta} m^{2\beta} f(m)\right].
    \end{align*}
    
    It means the function 
    \[
        g(n)=3^{3\beta} n^{2\beta} f(n)= 3^{3\beta} n^{2\beta} n^{3\alpha\log n} P(n) = 3^{18\alpha\log 3} n^{12\alpha\log 3} n^{ 3\alpha\log n} P(n)
    \]
    is weakly submultiplicative as in Definition \ref{def:weak-submulti}.
   
    By the extension of Fekete's lemma for $g(n)$, we have
    \[
        \lambda=\lim_{n\to\infty} \sqrt[n]{P(n)} = \lim_{n\to\infty} \sqrt[n]{g(n)} = \inf_n \sqrt[n]{g(n)}.
    \]
    It follows that $g(n)\ge \lambda^n$ for every $n$, which means
    \[
        P(n) \ge 3^{-18\alpha\log 3} n^{-12\alpha\log 3} n^{ -3\alpha\log n}\lambda^n.\qedhere
    \]
\end{proof}

\subsection*{Some remarks}
In principle, if we had 
\begin{equation} \label{eq:impossible-hypothesis}
    P(n)\le\const P(\ell)P(n-\ell)
\end{equation}
in Lemma \ref{lem:supporting-upper-bound} with a leading constant instead of a polynomial, we would obtain
\[
    P(n)\ge An^{-T}\lambda^n.
\]
However, \eqref{eq:impossible-hypothesis} seems to be not the case, since suppose Conjecture \ref{conj:order} is valid, we have $P(\ell)P(n-\ell)\le\const \ell^{-T} (n-\ell)^{-T}\lambda^n\le \const n^{-2T}\lambda^n$ when $\ell/n$ is in a constant range. Meanwhile, $P(n)\ge\const n^{-T}\lambda^n\ge \const n^T P(\ell)P(m)$. It means \eqref{eq:impossible-hypothesis} is valid only if $T=0$, which is quite unlikely.

Another idea to improve the bound in Theorem \ref{thm:lower-bound} may be a more specific $\ell$ in Lemma \ref{lem:supporting-upper-bound}. Suppose Conjecture \ref{conj:increasing} is valid, one can see that the value of $\ell$ that maximizes $P(\ell)P(n-\ell)$ is the value that maximizes the difference between $\ell$ and $n-\ell$. That is either $\ell= (n-1)/4$ or $\ell= (3n+1)/4$. It turns out to work, as in Section \ref{sec:conditional-lower-bound}.

\section{A conditional lower bound on $P(n)$}
\label{sec:conditional-lower-bound}
In this section, we prove Theorem \ref{thm:conditional-lower-bound}. We start with a useful observation.
\begin{observation}\label{obs:same-sum}
    If Conjecture \ref{conj:increasing} holds, then for any $\ell_1,m_1$ ($\ell_1\le m_1$) and $\ell_2,m_2$ ($\ell_2\le m_2$) with $\ell_1+m_1=\ell_2+m_2=n$ and $\ell_1\le\ell_2$, we have
    \[
        P(\ell_1)P(m_1)\ge P(\ell_2)P(m_2).
    \]
\end{observation}
\begin{proof}
As $\ell_2\ge\ell_1$, we write
\[
    \frac{P(\ell_2)}{P(\ell_1)} = \frac{P(\ell_2)}{P(\ell_2-1)} \frac{P(\ell_2-1)}{P(\ell_2-2)}\cdots \frac{P(\ell_1+1)}{P(\ell_1)}.
\]
Since $m_1=n-\ell_1\ge n-\ell_2=m_2$, we write
\[
    \frac{P(m_1)}{P(m_2)} = \frac{P(m_1)}{P(m_1-1)} \frac{P(m_1-1)}{P(m_1-2)}\cdots \frac{P(m_2+1)}{P(m_2)}.
\]
Suppose Conjecture \ref{conj:increasing} holds.
Since $m_1\ge \frac{n}{2}\ge \ell_2$, it follows from Conjecture \ref{conj:increasing} that for any $t$,
\[
    \frac{P(m_1-t)}{P(m_1-1-t)}\ge \frac{P(\ell_2-t)}{P(\ell_2-1-t)}.
\]
Since $m_1-m_2=\ell_2-\ell_1$, we have
\[
    \frac{P(m_1)}{P(m_2)}\ge \frac{P(\ell_2)}{P(\ell_1)}.
\]
The conclusion follows.
\end{proof}
Note that Observation \ref{obs:same-sum} can be seen as a generalization of Lemma \ref{lem:supermulti}, provided that Conjecture \ref{conj:increasing} holds, by setting $\ell_1=0$ and $m_1=n$.

We have the following corollary, which is stronger than Lemma \ref{lem:supporting-upper-bound}.
\begin{corollary} \label{cor:conditional-upper-bound} 
If Conjecture \ref{conj:increasing} holds, then for every $n$, we have
    \[
        P(n)\le n^3 P(\ell) P(n-\ell)
    \]
for $\ell=\lceil(n-1)/4\rceil$. 
\end{corollary}

We also need the following proposition, which corresponds to Proposition \ref{prop:supporting-upper-bound}. In fact, the analyzes in the proofs are quite similar.

\begin{proposition}\label{prop:conditional-upper-bound}
    If Conjecture \ref{conj:increasing} holds, then for every $n$ and any $0\le \ell\le n$, we have
    \[
        P(n)\le n^9 P(\ell) P(n-\ell).
    \]
\end{proposition}
\begin{proof}
    The conclusion holds for any $n<10$, as in the table below. Note that we only need to compare $P(n)$ with $P'(n)= n^9 P(\lfloor \frac{n}{2}\rfloor) P(\lceil \frac{n}{2}\rceil)$, due to Observation \ref{obs:same-sum}.
    \begin{center}
        \scriptsize
        \begin{tabular}{|c|c|c|c|c|c|c|c|c|c|}
            \hline
             n&1&2&3&4&5&6&7&8&9 \\
             \hline
             P(n) & 1& 2& 6& 19& 63& 216& 760 & 2725 & 9910 \\
             \hline
             $P'(n)$& 1 & 512 & 39366 & 1048576 & 23437500 & 362797056 & 4600311198 & 48452599808 & 463742325333\\
             \hline
        \end{tabular}
\end{center}

    Given some $n\ge 10$, suppose the conclusion holds for any number smaller than $n$. We show that it also holds for $n$. 
    
    We assume $\ell\le\frac{n}{2}$ since $\ell$ and $n-\ell$ are exchangeable.

    Let $k_1=\lceil\frac{n-1}{4}\rceil$. By Corollary \ref{cor:conditional-upper-bound}, we have
    \[
        P(n)\le n^3 P(k_1)P(n-k_1).
    \]
    
    If $\ell\le k_1$, then we are done since by Observation \ref{obs:same-sum}, we have
    \[
        P(n)\le n^3 P(k_1)P(n-k_1)\le n^3 P(\ell)P(n-\ell).
    \]
    We proceed with $\ell>k_1$. Let $k_2=\lceil \frac{1}{4} (n-k_1-1)\rceil$. Applying Corollary \ref{cor:conditional-upper-bound} to $P(n-k_1)$, we have
    \[
        P(n)\le n^3 P(k_1) P(n-k_1) \le n^3 P(k_1) (n-k_1)^3 P(k_2) P(n-k_1-k_2)\le n^6 P(k_1) P(k_2) P(n-k_1-k_2).
    \]
    If $\ell-k_1\le k_2$, then we are done since by Observation \ref{obs:same-sum}, we have
    \[
        P(n)\le n^6 P(k_1) P(k_2) P(n-k_1-k_2)\le n^6 P(k_1) P(\ell-k_1) P(n-\ell) \le n^6 P(\ell) P(n-\ell).
    \]
    We proceed with $\ell-k_1 > k_2$. Let $k_3 = \lceil\frac{1}{4} (n-k_1-k_2-1)\rceil$. Applying Corollary \ref{cor:conditional-upper-bound} to $P(n-k_1-k_2)$, we have
    \begin{align*}
        P(n)&\le n^6 P(k_1) P(k_2) P(n-k_1-k_2) \\
        &\le n^6 P(k_1) P(k_2) (n-k_1-k_2)^3 P(k_3) P(n-k_1-k_2-k_3)\\
        &\le n^9 P(k_1) P(k_2) P(k_3) P(n-k_1-k_2-k_3).
    \end{align*}
    We prove that $\ell-k_1-k_2 \le k_3$. Indeed,
    \begin{align*}
        k_3-(\ell-k_1-k_2) &\ge \frac{1}{4} (n-k_1-k_2-1) - \frac{n}{2} + k_1 + k_2\\
        &=\frac{3}{4} (k_1+k_2) - \frac{n}{4} - \frac{1}{4}\\
        &\ge \frac{3}{4}\left(\frac{1}{4}(n-1)+\frac{1}{4}\Bigl(n- \frac{n-1}{4} - 2\Bigr)\right) - \frac{n}{4} - \frac{1}{4}\\
        &= \frac{5}{64}n-\frac{49}{64}\\
        &\ge 0,
    \end{align*}
    since $n\ge 10$. (Note that $k_2=\lceil\frac{1}{4}(n-k_1-1)\rceil \ge\frac{1}{4}(n-\lceil\frac{n-1}{4}\rceil-1)\ge\frac{1}{4}(n-\frac{n-1}{4}-2)$.)

    Since $\ell-k_1-k_2\le k_3$, applying Observation \ref{obs:same-sum}, we have
    \begin{align*}
        P(n)&\le n^9 P(k_1) P(k_2) P(k_3) P(n-k_1-k_2-k_3)\\
        &\le n^9 P(k_1) P(k_2) P(\ell-k_1-k_2) P(n-\ell)\\
        &\le n^9 P(\ell)P(n-\ell).
    \end{align*}

    All cases are covered, and the conclusion follows by induction.
\end{proof}

Now comes the proof of Theorem \ref{thm:conditional-lower-bound}. It is quite similar to the proof of Theorem \ref{thm:lower-bound}.

\begin{proof}[Proof of Theorem \ref{thm:conditional-lower-bound}]
    By Proposition \ref{prop:conditional-upper-bound}, for every $\ell,m,n$ so that $\ell+m=n$, we have
    \[
        P(n)\le n^9 P(\ell) P(m).
    \]
    
    Consider a pair $\ell,m$ so that both $\ell,m$ are in the range $\left[\frac{n}{3}, \frac{2n}{3}\right]$, we can write
    \begin{align*}
        3^{18} n^9 P(n)\le 3^9 n^9 P(\ell) 3^9 n^9 P(m) \le 3^9 (3\ell)^9 P(\ell) 3^9 (3m)^9 P(m) = [3^{18} \ell^9 P(\ell)] [3^{18} m^9 P(m)].
    \end{align*}
   
    It means the sequence 
    \[
        g(n)=3^{18} n^9 P(n)
    \]
    is weakly submultiplicative as in Definition \ref{def:weak-submulti}.
  
    By the extension of Fekete's lemma for $g(n)$, we have
    \[
        \lambda=\lim_{n\to\infty} \sqrt[n]{P(n)} = \lim_{n\to\infty} \sqrt[n]{g(n)} = \inf_n \sqrt[n]{g(n)}.
    \]
    It follows that $g(n)\ge \lambda^n$ for every $n$, which means
    \[
        P(n) \ge 3^{-18} n^{-9} \lambda^n.\qedhere
    \]
\end{proof}

\section{A conditional upper bound on $\lambda$}
\label{sec:generating-function}
While the result in Theorem \ref{thm:lower-bound} may be interesting in the theoretical sense, the upper bound 
\[
    \sqrt[n]{3^{18\alpha\log 3} n^{12\alpha\log 3} n^{ 3\alpha\log n} P(n)}
\]
of $\lambda$ is quite poor given the available values of $P(n)$. 
Even the better constants in Theorem \ref{thm:conditional-lower-bound} still do not give good bounds.\footnote{The readers may notice that a great deal of effort has gone into optimizing the constants.} Therefore, we present another approach, which gives a conditional upper bound. The upper bound is valid when Conjecture \ref{conj:decreasing} is proved. Before presenting the conjecture, we give some related definitions first.

Given two polyominoes $X,Y$, we revisit the concatenation that proves the supermultiplicativity of $P(n)$ in Lemma \ref{lem:supermulti}: The largest cell of $X$ is placed right under the smallest cell of $Y$. The resulting polyomino can be seen as a composition of $X$ and $Y$. But now the concatenation introduces a class of polyominoes.

\begin{definition}
    A polyomino is said to be \emph{constructible} if it is the concatenation of two polyominoes of smaller sizes. And it is \emph{inconstructible} otherwise.
\end{definition}

Fig. \ref{fig:polyomino-inconstructible} gives an example of an inconstructible polyomino.
\begin{figure}[ht]
    \includegraphics[width=0.2\textwidth]{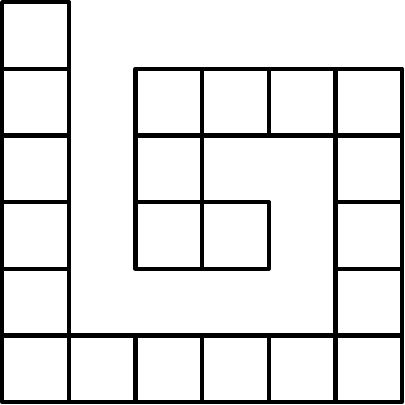}
    \caption{An inconstructible polyomino}
    \label{fig:polyomino-inconstructible}
\end{figure}

Let $Q(n)$ denote the number of inconstructible polyominoes of $n$ cells. The following relation was pointed out in \cite{rands1981animals}: For $n\ge 1$,
\[
    P(n) = \sum_{i=1}^n Q(i) P(n-i),
\]
where we set $P(0)=1$. We also set $Q(0)=1$ for later use.

In fact, the sequence $Q(n)$ can be deduced from the values of $P(n)$. Indeed, let $n_0$ be the largest number such that $P(1), \dots, P(n_0)$ are known, we can compute $Q(1), \dots, Q(n_0)$ because each $Q(n)$ can be computed from $P(1), \dots, P(n)$ and $Q(1), \dots, Q(n-1)$ by
\[
    Q(n) = P(n) - \sum_{i=1}^{n-1} Q(i) P(n-i).
\]

Note that in principle, we can change the definition of the concatenation of two polyominoes and obtain the corresponding definition of constructible polyominoes. The concatenation can be any so that the resulting polyomino is unique with no overlapping (the total number of cells is preserved). For example, we can define the concatenation of polyominoes $X$ and $Y$ by the polyomino obtained by translating them so that the largest cell of $X$ is placed just \emph{on the left} of the smallest cell of $Y$. In fact, $Q(n)$ does not really depend on how we define the concatenation of polyominoes. In any case, $Q(n)$ can be seen just as a function defined in terms of $P(n)$.

The available values of these sequences (with $n_0=56$) suggest the following conjecture.
\begin{conjecture} \label{conj:decreasing}
    The sequence $Q(n)/P(n)$ is monotonically decreasing.
\end{conjecture}
In fact, the sequence is not strictly decreasing with the only known exception $Q(2)/P(2)=Q(3)/P(3)=1/2$. (One may want to conjecture that the subsequence $Q(n)/P(n)$ for $n\ge 3$ is strictly decreasing, but we do not need strict monotonicity in this work.)

Suppose Conjecture \ref{conj:decreasing} holds, we have the following upper bound $U(n)\ge P(n)$:
\begin{equation} \label{eq:upper-bound-function}
    U(n)=
    \begin{cases}
        P(n) &\text{for $n\le n_0$}\\
        \sum_{i=1}^{n_0} Q(i) U(n-i) + \sum_{i=n_0+1}^n R U(i) U(n-i) &\text{otherwise}
    \end{cases}
\end{equation}
where $R=Q(n_0)/P(n_0)$.

The usage of dominating sequences is a common technique in bounding a not so well behaved sequence by a better behaved sequence. Particular applications to lattice animals can be found in \cite{barequet2019upper}, \cite{barequet2021concatenation}.

One can see $U(n)$ as a mixture of linear and bilinear recurrences with some starting values, where the values of $R$, $P(n)$ and $Q(i)$ are coefficients. Note that it is not really a pure recurrence since $U(n)$ appears in both sides. However, we can actually rewrite it as
\begin{align*}
    U(n)&=RU(n) + \sum_{i=1}^{n_0} Q(i) U(n-i) + \sum_{i=n_0+1}^{n-1} R U(i) U(n-i)\\
    \implies (1-R)U(n)&= \sum_{i=1}^{n_0} Q(i) U(n-i) + \sum_{i=n_0+1}^{n-1} R U(i) U(n-i)\\
    \implies U(n) &=\frac{1}{1-R} \left(\sum_{i=1}^{n_0} Q(i) U(n-i) + \sum_{i=n_0+1}^{n-1} R U(i) U(n-i)\right).
\end{align*}
We keep the former representation of $U(n)$ in \eqref{eq:upper-bound-function} for convenience later.

We verify that $U(n)$ is an upper bound of $P(n)$. Indeed, given some $n>n_0$ with $U(n')\ge P(n')$ for any $n'<n$, we have
\begin{align*}
    P(n)&=\sum_{i=1}^{n_0} Q(i)P(n-i) + \sum_{i=n_0+1}^{n} Q(i)P(n-i)\\
    &=Q(n) + \sum_{i=1}^{n_0} Q(i)P(n-i) + \sum_{i=n_0+1}^{n-1} \frac{Q(i)}{P(i)} P(i)P(n-i)\\
    &\le RP(n) + \sum_{i=1}^{n_0} Q(i)U(n-i) + \sum_{i=n_0+1}^{n-1} R U(i)U(n-i)\\
    \implies P(n)
    &\le \frac{1}{1-R} \left(\sum_{i=1}^{n_0} Q(i)U(n-i) + \sum_{i=n_0+1}^{n-1} R U(i)U(n-i)\right) =U(n).
\end{align*}

Calculating the sequence $U(n)$ for large $n$, it seems that $\sqrt[n]{U(n)}$ converges to something about $4.1$. We will investigate the generating function of $U(n)$, and from that deduce a bound $4.1141$ on the growth constant of $U(n)$ via the radius of convergence of the generating function. In particular, we have the following result.\footnote{Using the unpublished values of $P(n)$ for $n\le 70$, which are kindly shared by Gill Barequet in a private correspondence, one can even obtain $\lambda< 4.1038$.}

\begin{theorem} \label{thm:conditional-upper-bound}
    If Conjecture \ref{conj:decreasing} holds, we have
    \[
        \lambda < 4.1141.
    \]
\end{theorem}

In other words, if both Conjecture \ref{conj:increasing} and Conjecture \ref{conj:decreasing} are valid, we have a fairly narrow interval containing $\lambda$:
\[
    3.9909 < \lambda < 4.1141.
\]

\begin{remark}
    Note that the two conjectures are not directly related in the sense that neither conjecture implies the other conjecture if $P(n)$ is an arbitrary function. Indeed, consider the function $P'(n)$ satisfying Conjecture \ref{conj:increasing} with the starting values $P'(0)=1, P'(1)=1, P'(2)=2, P'(3)=7$, the corresponding function $Q'(n)$ with $Q'(0)=1, Q'(1)=1, Q'(2)=1, Q'(3)=4$ does not satisfy Conjecture \ref{conj:decreasing} since $Q'(3)/P'(3) > Q'(2)/P'(2)$. On the other hand, consider the functions $P''(n)$ and $Q''(n)$ satisfying Conjecture \ref{conj:decreasing} with the starting values $P''(0)=1, P''(1)=1, P''(2)=2, P''(3)=6, P''(4)=16$ and $Q''(0)=1, Q''(1)=1, Q''(2)=1, Q''(3)=3, Q''(4)=5$, we do not have Conjecture \ref{conj:increasing} for $P''$ since $P''(4)/P''(3) < P''(3)/P''(2)$.

    However, we may guess the approaches to settling the two conjectures could be similar. The matter is that neither conjecture appears to be an easy problem. At any rate, it is good to pursue Conjecture \ref{conj:decreasing} because the current best upper bound 
    \[
        \lambda < 4.5252
    \]
    has been only very recently established by Barequet and Shalah \cite{barequet2022improved}, after a long period of no improvement since the upper bound $\lambda < 4.6495$ by Klarner and Rivest \cite{klarner1973procedure} in 1973.
\end{remark}

\subsection*{Proof of Theorem \ref{thm:conditional-upper-bound}}
Let $f(x)=\sum_{n=0}^{\infty} U(n)x^n$ be the generating function of $U(n)$, we have the following proposition, whose proof in the appendix contains lengthy manipulations. 
\begin{proposition}\label{prop:genfunc-equation}
    \[
        R[f(x)]^2 + \left[\left(\sum_{i=0}^{n_0} (Q(i) - RP(i)) x^i\right)-2\right] f(x) + 1 = 0.
    \]
\end{proposition}

This is the equation $Af^2+Bf+C=0$ for $f=f(x)$ and 
\begin{align*}
A&=R,\\
B&=\left(\sum_{i=0}^{n_0} (Q(i) - RP(i)) x^i\right)-2,\\
C&=1.
\end{align*}

The discriminant is
\[
    \Delta(x)=B^2-4AC=\left[\left(\sum_{i=0}^{n_0} (Q(i) - RP(i)) x^i\right)-2\right]^2 - 4R.
\]
Therefore, 
\[
    f(x) = \frac{1}{2R} \left[2-\left(\sum_{i=0}^{n_0} (Q(i) - RP(i)) x^i\right) + \sqrt{\left[\left(\sum_{i=0}^{n_0} (Q(i) - RP(i)) x^i\right)-2\right]^2 - 4R}\right].
\]

We are not really interested in the actual function $f(x)$ but the condition of $x$ so that $f(x)$ is valid, that is $\Delta(x)\ge 0$. In fact, we need an estimate of the radius of convergence.

Let
\[
    g(x)=\sum_{i=0}^{n_0} (Q(i) - RP(i)) x^i.
\]

\begin{claim*}
    If $\theta>0$ is any value so that $g(\theta)\le 2-2\sqrt{R}$, then $\Delta(x)\ge 0$ for any $x\in[-\theta,\theta]$.
\end{claim*}
\begin{proof}
    It suffices to show that given such $\theta$, we have $g(x)\le 2-2\sqrt{R}$ for any $x\in[-\theta,\theta]$. It is because $g(x)\le 2-2\sqrt{R}$ implies $[g(x)-2]^2-4R\ge 0$, that is $\Delta(x)\ge 0$.
    
    Since
    \[
        Q(i) - RP(i) = P(i)\left(\frac{Q(i)}{P(i)} - \frac{Q(n_0)}{P(n_0)}\right) \ge 0,
    \]
    it follows that $g(x)$ is monotonically increasing in $[0,\infty)$. Therefore, $g(x)\le g(\theta)\le 2-2\sqrt{R}$ for $x\in[0,\theta]$. Meanwhile, also by the nonnegativity of the coefficients of $g(x)$, we have $g(-x)\le g(x)$ for any nonnegative $x$, that is $g(-x)\le g(x) \le g(\theta)\le  2-2\sqrt{R}$ for any $x\in[0,\theta]$. In total, $g(x)\le 2-2\sqrt{R}$ for any $x\in [-\theta,\theta]$.
\end{proof}

Using a computer program to do the calculations with $n_0=56$, we can find that $g(0.24307)\le 2-2\sqrt{R}$. We have tried to make the constant as large as reasonable, say we already have $g(0.24308) > 2-2\sqrt{R}$. Since $\Delta(x)\ge 0$ for $|x|\le 0.24307$, it follows that the radius of convergence is at least $0.24307$. In other words, we have a bound on the growth of $U(n)$:
\[
    \limsup_{n\to\infty} \sqrt[n]{U(n)} \le \frac{1}{0.24307} < 4.1141.
\]
Therefore, if Conjecture \ref{conj:decreasing} is valid, we have 
\[
    \lambda < 4.1141.
\]
\section*{Acknowledgments}
The author would like to thank Gill Barequet and an anonymous reviewer for their various useful remarks on the paper.

\section*{Competing interests}
The author declares that there is no competing interest.

\bibliographystyle{unsrt}
\bibliography{ana-klarner}

\appendix
\section{Proof of Proposition \ref{prop:genfunc-equation}}
We prove the equation of $f(x)$ in Proposition \ref{prop:genfunc-equation}.
At first,
\[
    [f(x)]^2=\left(\sum_{n=0}^{\infty} U(n)x^n\right) \left(\sum_{n=0}^{\infty} U(n)x^n\right) = \sum_{n=0}^{\infty} \sum_{i=0}^n U(i) U(n-i) x^n.
\]
Using this, we manipulate $f(x)$:
\begin{align*}
    f(x)&= \sum_{n=0}^{n_0} U(n) x^n + \sum_{n=n_0+1}^{\infty} U(n) x^n \\
    &= \sum_{n=0}^{n_0} P(n)x^n + \sum_{n=n_0+1}^{\infty} \left(\sum_{i=1}^{n_0} Q(i) U(n-i) + \sum_{i=n_0+1}^{n} RU(i)U(n-i)\right) x^n\\
    &= \sum_{n=0}^{n_0} P(n)x^n + \sum_{n=n_0+1}^{\infty} \sum_{i=1}^{n_0} Q(i) U(n-i) x^n + R \sum_{n=n_0+1}^{\infty} \sum_{i=n_0+1}^{n} U(i)U(n-i) x^n\\
    &= \sum_{n=0}^{n_0} P(n)x^n + \sum_{n=n_0+1}^{\infty} \sum_{i=0}^{n_0} Q(i) U(n-i) x^n - \sum_{n=n_0+1}^{\infty} U(n)x^n \\ 
        & \qquad + R\left([f(x)]^2 - \sum_{n=0}^{n_0} \sum_{i=0}^n U(i)U(n-i)x^n - \sum_{n=n_0+1}^{\infty} \sum_{i=0}^{n_0} U(i)U(n-i) x^n\right) \\
    &= \sum_{n=0}^{n_0} 2P(n)x^n - f(x) - \sum_{n=0}^{n_0}\sum_{i=0}^n RU(i)U(n-i)x^n\\
        & \qquad + \sum_{n=n_0+1}^{\infty} \sum_{i=0}^{n_0} (Q(i) - RU(i)) U(n-i) x^n + R[f(x)]^2.
\end{align*}

Note that
\begin{align*}
    &\sum_{n=n_0+1}^{\infty} \sum_{i=0}^{n_0} (Q(i) - RU(i)) U(n-i) x^n\\
    &= \sum_{i=0}^{n_0} (Q(i) - RU(i)) x^i \sum_{n=n_0+1}^{\infty} U(n-i) x^{n-i}\\
    &= \sum_{i=0}^{n_0} (Q(i) - RU(i)) x^i \left(f(x)-\sum_{n=0}^{n_0-i} U(n)x^n\right)\\
    &= \left(\sum_{i=0}^{n_0} (Q(i) - RU(i)) x^i\right) f(x) - \sum_{i=0}^{n_0} \sum_{n=0}^{n_0-i} (Q(i) - RU(i)) U(n)x^{n+i}.
\end{align*}
The latter summand can be rewritten as
\begin{align*}
    &\sum_{i=0}^{n_0} \sum_{n=0}^{n_0-i} (Q(i) - RU(i)) U(n)x^{n+i}\\
    &=\sum_{i=0}^{n_0} \sum_{n=i}^{n_0} (Q(i) - RU(i)) U(n-i)x^n \\
    &= \sum_{n=0}^{n_0} \sum_{i=0}^{n}  (Q(i) - RU(i)) U(n-i)x^n.
\end{align*}

It follows that
\begin{equation*}
    \begin{multlined}
        2f(x) = \sum_{n=0}^{n_0} 2P(n)x^n - \sum_{n=0}^{n_0}\sum_{i=0}^n RU(i)U(n-i)x^n + \left(\sum_{i=0}^{n_0} (Q(i) - RU(i)) x^i\right) f(x) \\
        - \sum_{n=0}^{n_0} \sum_{i=0}^{n}  (Q(i) - RU(i)) U(n-i)x^n + R[f(x)]^2.    
    \end{multlined}
\end{equation*}
Among them, we have
\begin{align*}
    & \sum_{n=0}^{n_0} 2P(n)x^n - \sum_{n=0}^{n_0}\sum_{i=0}^n RU(i)U(n-i)x^n - \sum_{n=0}^{n_0} \sum_{i=0}^{n}  (Q(i) - RU(i)) U(n-i)x^n\\
    &= \sum_{n=0}^{n_0} 2P(n)x^n - \sum_{n=0}^{n_0}\sum_{i=0}^n [RU(i)U(n-i) + (Q(i) - RU(i)) U(n-i)] x^n\\
    &= \sum_{n=0}^{n_0} 2P(n)x^n - \sum_{n=0}^{n_0}\sum_{i=0}^n Q(i) U(n-i) x^n\\
    &= \sum_{n=0}^{n_0} 2P(n)x^n - \sum_{n=0}^{n_0}\left[P(n) + \sum_{i=1}^n Q(i) U(n-i)\right] x^n\\
    &= \sum_{n=0}^{n_0} 2P(n)x^n - P(0)x^0 - \sum_{n=1}^{n_0}\left[P(n) + \sum_{i=1}^n Q(i) P(n-i)\right] x^n\\
    &= \sum_{n=0}^{n_0} 2P(n)x^n - P(0)x^0 - \sum_{n=1}^{n_0} 2P(n) x^n\\
    &= 2P(0)x^0 - P(0)x^0\\
    &= 1.
\end{align*}

In total, we obtain the conclusion
\begin{align*}
    2f(x) = 1 + \left(\sum_{i=0}^{n_0} (Q(i) - RP(i)) x^i\right) f(x) + R[f(x)]^2\\
    \implies R[f(x)]^2 + \left[\left(\sum_{i=0}^{n_0} (Q(i) - RP(i)) x^i\right)-2\right] f(x) + 1 = 0.
\end{align*}

\section{Supplementary program}
We verify $g(0.24307)\le 2-2\sqrt{R}$ and $g(0.24308) > 2-2\sqrt{R}$ in Section \ref{sec:generating-function} by the following Python $3$ program. The monotonicity of $P(n)/P(n-1)$ and $Q(n)/P(n)$ for the available values of $P(n), Q(n)$ is also checked. The \texttt{input} of the first $56$ values of $P(n)$ is obtained from Table $16.10$ in the book ``Polygons, Polyominoes and Polycubes'' edited by Guttmann.

%\lstinputlisting[language=Python]{ana-klarner.py}
\begin{lstlisting}[language=python]
#!/usr/bin/python
from fractions import Fraction
# input contains the available values P(1),P(2),...,P(56).
input = '''
1
2
6
19
63
216
760
2725
9910
36446
135268
505861
1903890
7204874
27394666
104592937
400795844
1540820542
5940738676
22964779660
88983512783
345532572678
1344372335524
5239988770268
20457802016011
79992676367108
313224032098244
1228088671826973
4820975409710116
18946775782611174
74541651404935148
293560133910477776
1157186142148293638
4565553929115769162
18027932215016128134
71242712815411950635
281746550485032531911
1115021869572604692100
4415695134978868448596
17498111172838312982542
69381900728932743048483
275265412856343074274146
1092687308874612006972082
4339784013643393384603906
17244800728846724289191074
68557762666345165410168738
272680844424943840614538634
1085035285182087705685323738
4319331509344565487555270660
17201460881287871798942420736
68530413174845561618160604928
273126660016519143293320026256
1088933685559350300820095990030
4342997469623933155942753899000
17326987021737904384935434351490
69150714562532896936574425480218'''

P=[1]+[int(line) for line in input.split()]
N=len(P)-1
print("n_0 =", N)
Q=[1]
for n in range(1, N+1):
    Q.append(P[n] - sum(Q[i]*P[n-i] for i in range(1,n)))

increasing=all(Fraction(P[n],P[n-1]) >= Fraction(P[n-1],P[n-2])
        for n in range(2,N+1))
print("P[n]/P[n-1] is increasing:", increasing)

decreasing=all(Fraction(Q[n],P[n]) <= Fraction(Q[n-1],P[n-1])
        for n in range(1,N+1))
print("Q[n]/P[n] is decreasing:", decreasing)

R=Fraction(Q[N],P[N])
def g(x):
    return sum((Q[i]-R*P[i]) * x**i for i in range(N+1))

for x in Fraction(24307,10**5), Fraction(24308,10**5):
    print("g(" + str(float(x)) + ") <= 2 - 2*sqrt(R):", \
            (g(x)-2)**2 >= 4*R and g(x) <= 2)
 
\end{lstlisting}

We use \texttt{Fraction} instead of \texttt{float} for the sake of an exact computation, since the divisions are made on integers only. (Note that the use of \texttt{float} in printing is only for a neat output.) The comparison is done manually with \texttt{(g(x)-2)**2 >= 4*R and g(x) <= 2} instead of using \texttt{g(x) <= 2-2*sqrt(R)} for the same purpose of precision.

Output of the program is:
\begin{lstlisting}
n_0 = 56
P[n]/P[n-1] is increasing: True
Q[n]/P[n] is decreasing: True
g(0.24307) <= 2 - 2*sqrt(R): True
g(0.24308) <= 2 - 2*sqrt(R): False
\end{lstlisting}

\end{document}